\documentclass[a4paper,12pt]{amsart}

\usepackage{amssymb}
\usepackage{latexsym} 
\usepackage{amsfonts} 
\usepackage{amsmath}
\usepackage{eucal} 
\usepackage{bm} 
\usepackage{bbm} 
\usepackage{graphicx} 
\usepackage[english]{varioref} 
\usepackage[nice]{nicefrac} 
\usepackage[all]{xy}
\usepackage{amsthm}

\newcommand{\End}{{\rm End}}
\newcommand{\Ind}{{\rm Ind}}
\newcommand{\res}{{\rm res}}
\newcommand{\Hom}{{\rm Hom}}

\def\ge{\geqslant}
\def\le{\leqslant}
\def\a{\alpha}
\def\b{\beta}
\def\g{\gamma}

\def\o{\omega}

\def\l{\lambda}

\def\i{^{-1}}

\newcommand{\kk}{\Bbbk}

\def\cl{\mathcal L}

\def\co{\mathcal O}

\theoremstyle{plain}
\newtheorem{thm}{Theorem}[section] 
\newtheorem*{thm*}{Theorem} 
 \newtheorem{prop}[thm]{Proposition}
 \newtheorem{lem}[thm]{Lemma}
 \newtheorem{cor}[thm]{Corollary}

\theoremstyle{definition}

\newtheorem{defn}[thm]{Definition}

\theoremstyle{remark}

\newtheorem*{rmk*}{Remark}
\newtheorem*{claim*}{Claim}

\begin{document}

\author{Xuhua He}
\address{Xuhua He, Department of Mathematics, The Hong Kong University of Science and Technology, Clear Water Bay, Kowloon, Hong Kong}
\thanks{Xuhua He was partially supported by HKRGC grants 601409.}
\email{maxhhe@ust.hk}
\author{Jesper Funch Thomsen}
\address{Institut for matematiske fag\\ Aarhus Universitet\\ 8000 \AA rhus C,
Denmark} \email{funch@imf.au.dk}

\title[]{On frobenius splitting of orbit closures of spherical subgroups in flag varieties}

\begin{abstract}
Let $H$ be a connected spherical subgroup of a semisimple algebraic group $G$. In this paper, we give a criterion for $H$-orbit closures in the flag variety of $G$ to have nice geometric and cohomological properties. Our main tool is the method of Frobenius splitting and of global F-regularity. 
\end{abstract}

\maketitle
\section{Introduction}

\subsection{} 
Let $G$ be a semisimple algebraic group and let $H$ denote a closed subgroup of 
$G$ acting with only finitely many orbits on the flag variety $\nicefrac{G}{B}$
associated with $G$. The group $H$ under this condition is called a spherical 
subgroup of $G$. The geometric and cohomological properties of the finitely
many $H$-orbit closures in $\nicefrac{G}{B}$ are of importance in representation theory. 

The case where $H$ is a Borel subgroup has been studied in great detail. The $H$-orbit closures are in
this case the set of Schubert varieties which have some remarkable properties: Schubert varieties are normal, Cohen-Macaulay and have rational singularities, all the higher cohomology groups of ample line bundles are zero, etc. They play an important role in  representation theory.  

Another important case is when $H$ is a symmetric subgroup; i.e. when $H$ is the set of fixed points of an involution of $G$. The classification and inclusion relation between the orbit closures have in this case been studied in great detail by Richardson and Springer  \cite{RS, RS2}. However, the singularities are considerably more complicated than in the case of Schubert varieties and the general picture is far from being fully understood. A non-normal example is constructed by Barbasch and Evens in \cite[6.9]{BE}. A non-normal, non-Cohen-Macaulay example for a spherical $H$ is constructed by Brion in \cite[Example 6]{B1}. 

\subsection{} In this paper, we give a criterion for $H$-orbit closures to have nice geometric and cohomological properties. Our main tool is the method of Frobenius splitting and of global F-regularity. 

The notion of Frobenius splitting was introduced by Mehta and Ramanathan in \cite{MR}. Any projective Frobenius split variety is weakly normal and the higher cohomology groups of ample line bundle are zero. 
The more restrictive notion of global F-regularity was recently introduced by K. Smith in \cite{S}. Any (projective) globally F-regular variety is normal and Cohen-Macaulay and the higher cohomology groups of  nef line bundles are zero. 

The main result can be briefly stated as follows

\begin{thm}
Let $H$ be a connected reductive subgroup of $G$ and $B$ denote a Borel subgroup of $G$ such that $B_H=B \cap H$ is a Borel subgroup of $H$. Assume furthermore that $(G, H)$ is a Donkin pair or that the characteristic of the ground field $\kk$ is sufficiently large. Let $J$ be a subset of the set $I$ of simple roots of $G$ and 
let $\rho_J$ (resp. $2\rho_H$) denote the sum of the fundamental weights within $J$ (resp. the sum of the positive roots of $H$). Then

(1) If $2 \rho_H-\rho_J$ is dominant for $B_H$, then $\nicefrac{H P_J}{B}$ admits a Frobenius 
 splitting along an ample divisor that is compatible with all subvarieties of the form $\nicefrac{\overline{H B w B}}{B}$ for $w \in W_J$. 

(2) If moreover $2 \rho_H-\rho_J$ is dominant regular for $B_H$, then $\nicefrac{\overline{H B w B}} {B}$ is globally F-regular for all $w \in W_J$. 
\end{thm}

Notice that we do not assume $H$ to be spherical subgroup in the above theorem. 
However, in many cases the relevant orbit closures $\nicefrac{\overline{H B w B}}{B}$ 
coincide with closures of orbits under spherical subgroups.
For example, if $H$ is the trivial subgroup of $G$ then the theorem applies for $J=I$. In 
this case the varieties $\nicefrac{\overline{H B w B}}{B}$ , for $w \in W$, are 
just the set of  Schubert varieties.  In particular, in this way we obtain the well 
known results that the flag variety admits a Frobenius splitting along an ample divisor which 
is compatible with all Schubert varieties and that any Schubert variety is globally F-regular.

Another special case is when $(G, H)$ is in N. Ressayre's list of minimal rank pairs \cite{Re}. In 
this case one finds that the flag variety admits a Frobenius splitting along an ample divisor that is compatible 
with all the $H$-orbit closures. 

Notice that one cannot expect the flag variety $\nicefrac{G}{B}$  to be  Frobenius split compatible
with all the $H$-orbit closures for a given spherical subgroup $H$ of $G$.
For example, if $(G, H)=(SL_n, SO_n)$, then the scheme theoretic intersection of two
$H$-orbit closures might not be a reduced scheme. Thus the desired Frobenius splitting cannot exist. 
For more details, see \cite[Introduction]{B1}. 

\subsection{} Let us make a short digression and discuss another criterion for $H$-orbit closures to 
have nice properties. 

In \cite{B1}, Brion introduced {\it multiplicity-free} subvarieties of the flag variety. A subvariety is multiplicity-free if it is rationally equivalent to a linear combination of Schubert cycles with coefficients equal to either  $0$ or $1$. 
In \cite{Br} Brion proved that multiplicity-free subvarieties are normal, Cohen-Macaulay and have nice cohomological properties. 

In a recent work \cite{K}, Knutson proved that given a multiplicity-free divisor $X$ of the flag variety, there exists a Frobenius splitting on the flag variety that is compatible with $X$. It is still unknown if any multiplicity-free subvariety admits a Frobenius splitting. 

The applications of the results in this paper include many multiplicity-free cases, but they also include some non multiplicity-free cases. See the Example in Section \ref{Ex2} and \ref{Ex3}. It is interesting to compare the criterion in this paper with the multiplicity-free criterion. 

\subsection{} The paper is organized as follows. In Section 2 we introduce notation. In Section 3 we give a short introduction to Frobenius splitting and global F-regularity. The main technical result (Theorem \ref{thm1}) is presented in Section 4. In Section 5, we discuss the surjectivity condition appearing in Theorem \ref{thm1}. In Section 6, we discuss some Frobenius splitting of the flag variety $\nicefrac{P_J}{B}$, which will be used in Section 7. In Section 7, we prove the main result of this paper and discuss some applications. In Section 8, we discuss some examples and non examples.

\section{Notation}

\subsection{}Throughout this paper $G$ will denote a connected semisimple and simply connected
linear algebraic group over an algebraically closed field
$\kk$. Within $G$ we will fix a Borel subgroup $B$ and a maximal torus $T \subset B$. 

The set of roots of $G$ determined by $T$ will be denoted by $R$ and the set of 
positive roots determined by $(B, T)$ will be denoted by $R^+$. The simple roots are 
denoted by $\a_i$, ${i \in I}$. For $i \in I$, let $s_i$ be the corresponding simple reflection and $\o_i$ be the corresponding fundamental weight. We let $\rho$ denote half the sum of the positive roots
or, alternatively defined, the sum of the fundamental weights.

The Weyl group $W=\nicefrac{N_G(T)}{T}$ is generated by the simple reflections $s_i$, 
for $i \in I$. The length of $w \in W$ will be denoted by $l(w)$, and the 
element of maximal length will be denoted by $w_0$. By abuse of notation 
$w$ will sometimes both denote an element $w \in W$ and a corresponding 
element within the normalizer $N_G(T)$. The set of Schubert varieties
in $\nicefrac{G}{B}$ is indexed by the elements in $W$. We use the notation $X(w)$
for the Schubert variety defined as the closure of $\nicefrac{BwB}{B}$.

\subsection{}

For $J \subset I$, let $P_J \supset B$ be the corresponding standard
parabolic subgroup and $L_J \supset T$ be the corresponding Levi
subgroup of $P_J$. Let $U_{J}$ be the unipotent radical of $P_J$. 
Let $W_J$ denote the parabolic subgroup of $W$ generated by $s_j$ 
for $j \in J$ and $w_0^J$ denote the element of maximal length in $W_J$. 
Let $\rho_{J}=\sum_{j \in J} \o_j$. The set of  positive roots determined 
by $B \cap L_J$ in $L_J$ is denoted by $R_J^+$.

\subsection{}

By $H$ we will denote a connected  reductive subgroup of $G$. 
We will assume that $B$ and $T$ are chosen such that 
$B_H = H \cap B$ is a Borel subgroup of $H$ and $T_H = H \cap T$ 
is a maximal torus of $H$.

The roots $R_H$ of $H$ determined by $T_H$ is the set of nonzero 
restrictions of the roots in $R$. We consider the character group  
$ X^*(T_H) $ of $T_H$ as embedded inside the 
tensorproduct $X^*(T_H)_{\mathbb Q}=  X^*(T_H) \otimes_{\mathbb Z} {\mathbb Q}.$
By $\rho_H \in X^*(T_H)_{\mathbb Q}$ we then denote half the sum of 
the positive roots of $H$.

\subsection{}

For any integral weight $\l$ of $T$, let $\kk_{-\l}$ be the
one-dimensional representation of $B$ with weight $-\l$ and 
$\cl(\l)=G \times_B \kk_{-\l}$ be the corresponding $G$-linearized 
line bundle on $\nicefrac{G}{B}$. Let 
$$\nabla(\l)=\Ind^G_B(\kk_{-\l})={\rm H}^0(\nicefrac{G}{B}, \cl(\l)),$$ 
denote the dual Weyl $G$-module with lowest weight $-\l$ (if $\lambda$
is dominant). The restriction of $\mathcal L(\lambda)$ to $\nicefrac{P_J}{B}$
will be denoted by $\mathcal L_J(\lambda)$ and we define 
$$\nabla^J(\l) =  \Ind^{P_J}_B(\kk_{-\l})={\rm H}^0(\nicefrac{P_J}{B}, \cl_J(\l)) .$$
When $\nu$ is an integral $T_H$-weight we similarly
write $\cl_H(\nu)=H \times_{B_H} \kk_{-\nu}$ and 
$$ \nabla^H(\nu) = \Ind^{H}_{B_H}(\kk_{-\nu})
= {\rm H}^0(\nicefrac{H}{B_H}, \cl_H(\nu)).$$ 

When $\kk$ is a field of positive characteristic $p>0$ then 
the $G$-module ${\rm St} = \nabla((p-1)\rho)$ will play a special 
role. This module is called the Steinberg module. The 
Steinberg module is known to be an irreducible and 
self-dual $G$-module. When it makes sense we let ${\rm St}_H$ denote the 
Steinberg module of $H$.

\subsection{}
By a variety we mean a reduced and separated scheme of finite type over $\kk$. In particular, we allow a variety to have several irreducible components. 
When $X$ is a $B_H$-variety we define an action of $B_H$ on 
$H \times X$ by $b \cdot (h, x)=(h b \i, b \cdot x)$. The
quotient is then denoted by $H \times_{B_H} X$ or sometimes
by $X_H$. This way we obtain an equivalence between the 
set of quasi-coherent $B_H$-linearized sheaves 
on $X$ and quasi-coherent $H$-linearized sheaves on $X_H$.
The $H$-linearized sheaf on $X_H$ corresponding to a 
$B_H$-linearized sheaf $\mathcal F$ on $X$ is denoted 
by ${\mathcal Ind}_{B_H}^H (\mathcal F)$. This notation
is  explained by the $H$-equivariant identity
$$ {\rm H}^0\big( X_H, {\mathcal Ind}_{B_H}^H (\mathcal F)
\big) \simeq  {\rm Ind}_{B_H}^H \big( {\rm H}^0 \big(
X, \mathcal F) \big).$$
The sheaf ${\mathcal Ind}_{B_H}^H (\mathcal F)$ is characterized
as the $H$-linearized sheaf satisfying that its $B_H$-linearized
restriction to 
$$ X \simeq \{ e \} \times X \subseteq H \times_{B_H} X,$$
is $\mathcal F$.

\section{Frobenius splitting}

\subsection{}\label{fsplit} In this section $\kk$ denotes an algebraically closed field of positive characteristic $p$. 
Let $X$ be a scheme of finite type over $\kk$. The {\it absolute Frobenius morphism} $F : X \rightarrow X$ on $X$ is the morphism of schemes which on the level of points is the identity map and where the associated map of sheaves $$F^\sharp : \mathcal O_{X} \rightarrow F_* \mathcal O_{X},$$ is the $p$-th power map. 
Define $\End_F(X)$ to be the $\kk$-vector space which as a abelian group equals $\Hom_{\co_{X}}\big(F_* 
\mathcal O_X, \mathcal O_{X} \big)$, but where the $\kk$-structure is twisted by the map
$a \mapsto a^{\frac{1}{p}}$.  A {\it Frobenius splitting} of $X$ is then an element $s$ in $\End_F(X)$ such 
that the composition $s \circ F^\sharp$ is the identity map.

\subsection{}

Let $\kk[X]$ denote the space of global regular functions on $X$. The 
evaluating of an element $s$ of $\Hom_{\co_{X}}\big(F_* 
\mathcal O_X, \mathcal O_{X} \big)$
at the constant global function $1$ on $X$ defines an element in $\kk[X]$.
This way we obtain a  $\kk$-linear map 
\begin{equation}
\label{ev}
 \Hom_{\co_{X}}\big(F_* 
\mathcal O_X, \mathcal O_{X} \big)  \rightarrow \kk[X] .
\end{equation}
Composing (\ref{ev}) with the Frobenius morphism on 
$\kk[X]$ is then a  $\kk$-linear map
\begin{equation}
{\rm ev}_X : \End_F(X) \rightarrow \kk[X],
\end{equation}
which is called the {\it evaluation map}. 
\subsection{}

Let $\mathcal M$ denote a line bundle on $X$ and define $\End^{\mathcal M}_F(X)$
to be the $\kk$-vector space which as an abelian group is $\Hom_{\co_{X}}\big(F_* 
\mathcal M, \mathcal O_{X} \big)$ but where the $\kk$-structure is twisted by the
 map $a \mapsto a^{\frac{1}{p}}$. A {\it Frobenius $\mathcal M$-splitting} of $X$ 
is an element $s_{\mathcal M}$ of $\End^{\mathcal M}_F(X)$ for which there 
exists a global section $\frak m$ of $\mathcal M$ such that the composed map
\begin{equation} 
\label{lsplit}
F_* \mathcal O_X \xrightarrow{F_* \frak m} F_* \mathcal M 
\xrightarrow{s_{\mathcal M}} \mathcal O_X,
\end{equation}
defines a Frobenius splitting. The construction of (\ref{lsplit}) 
from $\mathfrak m$ and $s_{\mathcal M}$ is a special case of a 
general $\kk$-linear morphism 
\begin{equation}
\label{canmap}
\End^{\mathcal M}_F(X) \otimes {\rm H}^0\big(X , \mathcal M \big)
\rightarrow \End_F(X) .
\end{equation}
Notice that for (\ref{canmap}) to be $\kk$-linear it is necessary that 
the $\kk$-structure on $\End^\mathcal M_F(X)$ is chosen in the
given way.

In case $X$ is a smooth variety one has the following canonical $\kk$-linear 
identification
\begin{equation}
\End^{\mathcal M}_F(X) \simeq {\rm H}^0\big(X, \omega_X^{1-p} \otimes 
\mathcal M^{-1} \big),
\end{equation}
where $\omega_X$ denotes the dualizing sheaf of $X$.
In this setting the map (\ref{canmap}) is just the 
multiplication map 
$${\rm H}^0\big(X , \omega_X^{1-p} \otimes  \mathcal M^{-1} \big) \otimes 
{\rm H}^0\big(X , \mathcal M \big)
\rightarrow
{\rm H}^0\big(X , \omega_X^{1-p}  \big).
$$

\subsection{}

Let $Y$ denote a closed subscheme of $X$ with sheaf ideals denoted 
by $\mathcal I_Y$. The subvector space of $\End^{\mathcal M}_F(X)$ 
consisting of the elements $s_\mathcal M$ satisfying
$$ s_{\mathcal M} \big(   \mathcal M  \otimes
\mathcal  I_Y \big) \subseteq  \mathcal I_Y,$$
is denoted by $\End^{\mathcal M}_F(X,Y)$. If $X_i$, $i \in \mathfrak I$,
is a collection of closed subschemes of $X$ then we use the notation 
 $\End^{\mathcal M}_F(X,\{ X_i\}_{i \in  \mathfrak I})$ for the intersection 
of the $\End^{\mathcal M}_F(X,X_i)$ for $i \in \mathfrak I$. 
When 
$\mathcal M = \mathcal O_X$ we remove $\mathcal M$ from all 
of the above notation.

\subsection{}

Let $s_\mathcal M$ in  $\End^{\mathcal M}_F(X,\{ X_i\}_{i \in  \mathfrak I})$
define a Frobenius $\mathcal M$-splitting of $X$. In this case we say that $s_\mathcal M$
is compatible with the subschemes $X_i$, $i \in \mathfrak I$. The following result is 
standard (see e.g. \cite[Lemma 3.1]{HT2}) 

\begin{lem}
\label{compatible1}
Let $Y$ and $Z$ be closed subvarieties 
in $X$ and let $s$ be a global section 
of $\End_F^{\mathcal M} (X,\{ Y,Z \})$.
\begin{enumerate}
\item $s_{\mathcal M} \in \End_F^{\mathcal M}(X,Y_1)$ for every 
irreducible
component $Y_1$ of $Y$.
\item If $Y \cap Z$ denotes the scheme theoretic intersection
then $s_{\mathcal M}$ is contained
in $\End_F^\cl(X,Y \cap Z)$.
\end{enumerate}
\end{lem}

\subsection{}

Let $R$ denote a localizations of a finitely generated $\kk$-algebra
and assume, for simplicity, that $R$ is an integral domain. In 
the following we use the  notation $F_*R^e$, $e \in \mathbb N$, 
to denote the $R$-module which as an abelian group is just 
$R$ but where the $R$-structure is twisted by the iterated
Frobenius map $r \mapsto r^{p^e}$. 

The following notion was introduced by M. Hochster and
C. Huneke.

\begin{defn}
The ring $R$ is said to be {\it strongly $F$-regular} if for 
each $r \in R$ there exists an $e \in \mathbb N$  and an
$R$-linear map
$$ F_*^e R \rightarrow R,$$
which maps $r$ to $1$. 
\end{defn}

Strongly $F$-regular rings have nice geometric properties; e.g. 
they are normal and Cohen-Macaulay.
It is known that a ring $R$ is strongly $F$-regular if and 
only if all its local rings are strongly $F$-regular. We 
define an irreducible  variety $X$ to be 
strongly $F$-regular if all its local rings are strongly 
$F$-regular. In that case the coordinate ring of any open 
affine subvariety of $X$ is also strongly $F$-regular.  
The  Schubert varieties $X(w)$ are examples of 
strongly $F$-regular varieties  \cite{LRT}.
 
We now recall the following important notion introduced by Karen Smith
\cite{S}. 

\begin{defn}
Let $X$ denotes an irreducible projective variety and $\mathcal M$ 
denote a ample line bundle on $X$. If the {\it section ring }
$$ \bigoplus_{n \geq 0} {\rm H}^0 \big( X, \mathcal M^n \big),$$
is strongly $F$-regular then $X$ is said to be {\it globally $F$-regular}.
\end{defn}

It should be noticed that the definition above is independent of the 
chosen ample line bundle $\mathcal M$. For later use we observe 
the following result \cite[Thm.3.10]{S}

\begin{lem}
\label{global-lem}
Let $X$ denote an irreducible projective and strongly $F$-regular 
variety. If $X$ admits a Frobenius $\mathcal M$-splitting by an 
ample line bundle $\mathcal M$, then $X$ is globally $F$-regular.  
\end{lem}

Another useful fact, observed in \cite[Lemma 1.2]{LRT}, is the following

\begin{lem}
\label{glob-push}
Let $f : X \rightarrow Y $ denote a morphism of projective varieties. 
Assume that  the induced map
$$ \mathcal O_Y \rightarrow f_* \mathcal O_X,$$
is an isomorphism and that $X$ is globally $F$-regular. Then 
$Y$ is also globally $F$-regular.
\end{lem} 

To apply this result we will later use the following fact

\begin{lem}
\label{pushforward}
Let $f : X \rightarrow Y $ denote a surjective 
morphism of projective varieties. Let $X'$ denote a 
closed subvariety of $X$ and let $Y' = f(X')$ denote 
its image. Assume that the map 
$$ \mathcal O_Y \rightarrow f_* \mathcal O_X,$$
induced by $f$, is an isomorphism and let $\mathcal M$ 
denote an ample line bundle on $X$. If $ X$ 
admits a Frobenius $\mathcal M$-splitting compatible
with $X'$ then the map 
$$ \mathcal O_{Y'} \rightarrow f_* \mathcal O_{X'},$$
induced by $f$, is also an isomorphism.
\end{lem}
\begin{proof}
Let $\mathcal L$ denote an ample line bundle 
on $Y$. By \cite[Lemma 3.3.3(b)]{BK} it suffices to
prove that the induced morphism
$$ {\rm H}^0 \big( Y', \mathcal L^n \big ) 
\rightarrow {\rm H}^0 \big ( X', f^* \mathcal  L^n \big),$$
is surjective for sufficiently large $n$. By the 
assumption the corresponding statement on 
$X$ and $Y$ are satisfied. 
Commutative of the following diagram 
\begin{equation}
\xymatrix{
{\rm H}^0 \big( Y, \mathcal L^n \big )  \ar[r]^{f^*}  \ar[d]_{res^Y_{Y'}} & 
{\rm H}^0 \big( X, f^* \mathcal L^n \big ) \ar[d]^{res^X_{X'}} \\
{\rm H}^0 \big( Y', \mathcal L^n \big )  \ar[r]^{f^*} & {\rm H}^0 \big( X', f^* \mathcal L^n \big ) 
}
\end{equation}
then implies that it suffices to prove that the restriction map 
$${\rm H}^0 \big( X, f^* \mathcal L^n \big ) \rightarrow 
{\rm H}^0 \big( X', f^* \mathcal L^n \big ), $$
is surjective. As  $f^* \mathcal L^n$ is globally generated
this follows by the ampleness of $\mathcal M$ and the 
assumption that $X$ admits a Frobenius 
$\mathcal M$-splitting compatible with $X'$ 
(cf. \cite[Thm.1.4.8(ii)]{BK}).
\end{proof}

\section{Frobenius splitting of $H \times_{B_H} X$}
\label{FsplitX}

In this section we let $\mathcal M$ denote a $B$-linearized line 
bundle on an irreducible projective $B$-variety $X$. The notation 
$\mathcal M_H$ is used  to denote the corresponding 
$H$-linearized line bundle on $X_H = H \times_{B_H} X$. 
A linearized line bundle on  $\nicefrac{H}{B_H}$ is determined by a 
$B_H$-character $\nu$. The pull-back of such a line bundle 
to $H \times_{B_H} X$ is denoted by $\mathcal O(-\nu)$. The
tensor product of  $\mathcal O(\nu)$ and $\mathcal M_H$ will
be denoted by $\mathcal M_H(\nu)$.
Recall that 
we have an $H$-equivariant identification 
\begin{equation}
\label{id}
{\rm H}^0 \big( X_H, \mathcal M_H(\nu)  \big) \simeq 
{\rm Ind}_{B_H}^H \big( {\rm H}^0 \big(X, \mathcal M \big) \otimes
\kk_{-\nu}\big).
\end{equation}

We will now consider the following setup  : let
$\lambda$ denote a dominant weight of $G$
and $X_i$, $i \in \mathfrak I$, denote a collection of closed 
subvarieties in $X$. By 
\begin{equation}
\label{theta}
\theta :  \kk_{\lambda } \rightarrow
 {\rm End}_F^{\mathcal M} \big( X,\{ X_i\} \big),
\end{equation}
and
\begin{equation}
\label{phi}
\phi  : \nabla(\l)  \rightarrow {\rm H}^0 \big ( X, \mathcal M \big).
\end{equation}
we denote $B$-equivariant maps. By 
$$(\theta,\phi) : \nabla(\l) \otimes \kk_{\lambda }
\rightarrow {\rm End}_F^{} \big( X,\{ X_ i\} \big),$$
we denote the map induced by $\theta$, $\phi$ and the natural $B$-equivariant 
morphism
$${\rm End}_F^{\mathcal M} \big( X,\{ X_i\} \big) \otimes 
{\rm H}^0 \big ( X, \mathcal M \big) \rightarrow 
{\rm End}_F^{} \big( X,\{ X_ i\} \big).$$
Then we can formulate

\begin{thm}
\label{thm1}
Assume that the following conditions are satisfied 
\begin{enumerate}
\item $(\theta,\phi)$ contains a Frobenius splitting of $X$
in its image.
\item The $T_H$-character $2 (p-1)\rho_H  - \lambda_{|T_H}$ is dominant.
\item The $H$-equivariant restriction morphism 
$$  \nabla \big( \lambda  \big)  \rightarrow \nabla^H \big( \lambda_{|T_H} \big),$$
is surjective.
\end{enumerate}
Then $H \times_{B_H} X$ admits a Frobenius $\mathcal M_H(
2 (p-1) \rho_H - \lambda_{|T_H} )$-splitting 
which is compatible with the subvarieties $H \times_{B_H} X_i$, $i \in
\mathcal I$. 
\end{thm}
\begin{proof}
Recall that the $B_H$-character associated to the dualizing sheaf 
on $\nicefrac{H}{B_H}$ equals $2 \rho_H$. In particular, we have 
an $H$-equivariant identification (cf. \cite[Sect. 5]{HT2})
$$ {\rm End}_F^{\mathcal M_H( \nu)} (X_H, \{ (X_i)_H \} )
\simeq {\rm Ind}_{B_H}^H \big ( {\rm End}_F^{\mathcal M} (X, \{ X_i \})
\otimes \kk_{-\lambda}  \big ),$$
where  $\nu$ denotes the $B_H$-character $2 (p-1) \rho_H - \lambda_{|T_H} $, and where we have used the notation $(X_i)_H$ for $H \times_{B_H} X_i$.
By application of Frobenius reciprocity this show that $\theta$ 
induces an $H$-equivariant morphism
\begin{equation}
\label{ind-phi}
\kk \rightarrow  {\rm End}_F^{\mathcal M_H( \nu)} (X_H, \{ (X_i)_H \} ).
\end{equation}
Let $g$ denote a nonzero element inside the 
image of (\ref{ind-phi}). It then suffices to find a global section $s$
of ${\mathcal M_H( \nu)}$ such that $ g(s)$
is nonzero.

To find $s$ we first apply Frobenius reciprocity to (\ref{phi})
to obtain
\begin{equation}
\label{ind-theta0}
\nabla(\l) \rightarrow {\rm Ind}^H_{B_H} \big(  {\rm H}^0 \big ( X, \mathcal M \big) \big)
 \simeq {\rm H }^0  \big( X_H,  \mathcal M_H\big),
\end{equation}
which combined with the natural morphism
\begin{equation}
\label{nu}
\nabla^H(\nu) \rightarrow  {\rm H }^0  \big( X_H,  \mathcal O(\nu)\big),
\end{equation}
defines an $H$-equivariant morphism
\begin{equation}
\label{ind-theta}
\nabla(\l) \otimes \nabla^H(\nu)  \rightarrow {\rm H}^0 \big ( X_H, \mathcal M_H(\nu) \big).
\end{equation}We claim that we can find the desired $s$ inside the image of 
(\ref{ind-theta}).

Applying the natural morphism 
$$ {\rm End}_F^{\mathcal M_H( \nu)} (X_H, \{ (X_i)_H \} )
\otimes {\rm H }^0  \big( X_H,  \mathcal M_H\big)
\rightarrow {\rm End}_F^{\mathcal O( \nu)} (X_H, \{ (X_i)_H \} ),
$$
we see that $g$ and (\ref{ind-theta0}) defines an $H$-equivariant
map
\begin{equation} 
\label{map}
\Phi :
\nabla (\l) \rightarrow  {\rm End}_F^{\mathcal O( \nu)}
 (X_H, \{ (X_i)_H \} ) .
\end{equation}
By construction (\ref{map}) is the map induced by 
$(\theta,\phi)$ and the identification
$${\rm End}_F^{\mathcal O( \nu)}  (X_H, \{ (X_i)_H \} )
\simeq  {\rm Ind}_{B_H}^H \big ( {\rm End}_F (X, \{ X_i \})
\otimes \kk_{-\l}  \big ),$$
in particular, the following diagram is commutative
\begin{equation}
\label{com1}
\xymatrix{
\nabla (\l) \ar[r]^(.25){\Phi} \ar[dr]_(0.35){(\theta,\phi) \otimes \kk_{-\l} } &  {\rm Ind}_{B_H}^H \big ( {\rm End}_F (X, \{ X_i \})
\otimes \kk_{-\l}  \big )  \ar[d] \ar[rr]^(.7){{\rm Ind}_{B_H}^H ({\rm ev}_X) } & & \nabla^H(\l_{|T_H}) \ar[d] \\
&  \End_F(X, \{ X_i\} ) \otimes \kk_{-\l}  \ar[rr]^(.6){{\rm ev}_X} & & \kk_{-\l}
}
\end{equation}
Notice that the lower horizontal part of the diagram (\ref{com1}) 
is a $B$-equivariant morphism and thus it must, up to a constant, coincide 
with the
projection map onto the lowest weight space of $\nabla (\l)$.
By assumption (1) this map is  nonzero.  Thus the
composed upper  horizontal morphism $\nabla (\l) \rightarrow 
\nabla^H (\l_{|T_H})$ must, up to a non-zero constant, be the natural 
restriction map. In particular,  the composed upper horizontal 
map  $\nabla (\l) \rightarrow \nabla^H (\l_{|T_H})$ is surjective by 
assumption (3).

Consider next the natural morphism 
\begin{equation}
\label{map1}
{\rm End}_F^{\mathcal O( \nu)} (X_H, \{ (X_i)_H \} )
\otimes {\rm H }^0  \big( X_H,  \mathcal O(\nu) \big)
\rightarrow {\rm End}_F (X_H, \{ (X_i)_H \} ).
\end{equation}
Composing (\ref{map}) with (\ref{nu}) and $\Phi$
defines the following commutative diagram
\begin{equation}
\label{com2} 
\xymatrix{
\nabla(\l) \otimes \nabla^H(\nu)  
\ar[dr]^\psi \ar[d]_{\Phi \otimes {\bf 1}} 
& & \\
 {\rm End}_F^{\mathcal O( \nu)}
(X_H, \{ (X_i)_H \} )  \otimes \nabla^H(\nu) 
\ar[d]_{{\rm  Ind}_{B_H}^H({\rm ev}_X  ) \otimes {\bf 1}}
 \ar[r]
  & 
{\rm End}_F (X_H, \{ (X_i)_H \} )
 \ar[dr]^(.7){\rm ev_{X_H} }  
\ar[d]_{{\rm Ind}_{B_H}^H( {\rm ev}_X ) }
& \\
\nabla^H(\l_{|T_H}) \otimes \nabla^H(\nu) 
\ar[r]_m
& \nabla^H(2 (p-1) \rho_H ) 
\ar[r]_(.7){{\rm ev}_{\nicefrac{H}{B_H} }} 
& \kk
}
\end{equation}
where $m$ is the multiplication map and $\psi$ is the map
induced by $g$. In this notation we have to show that the image 
of $\psi$ 
contains a Frobenius splitting or, equivalently, that the 
map 
$$ \nabla(\l) \otimes \nabla^H(\nu) \rightarrow
\kk,$$
from the upper left corner in (\ref{com2}) to the 
lower right corner is surjective. First of all the 
composed vertical map in (\ref{com2}) is surjective
by the above observations. Moreover, $m$ is surjective by 
assumption (2) and \cite[Thm. 3]{RR} while ${\rm ev_{\nicefrac{H}{B_H}} }$ is surjective because 
$\nicefrac{H}{B_H}$ admits a  Frobenius splitting.
This ends the proof.
\end{proof}

\begin{rmk*} 
\label {remark}
The statement of Theorem \ref{thm1} provides
us with a morphism 
\begin{equation}
\label{Lsplit}
F_*\mathcal M_H\big(
2 (p-1) \rho_H  - \lambda_{|T_H} \big) 
\xrightarrow{g} \mathcal O_{X_H},
\end{equation}
and a global section $s$ of the line bundle 
$\mathcal M_H\big(
2 (p-1) \rho_H  - \lambda_{|T_H} \big)$
such that the composition of (\ref{Lsplit})
with 
\begin{equation}
\label{Lsplit1}
F_* \mathcal O_{X_H}
\xrightarrow{F_* s} F_*\mathcal M_H\big(
2 (p-1) \rho_H  - \lambda_{|T_H} \big),
\end{equation}
defines a Frobenius splitting of $X_H$.  Actually
the proof of Theorem \ref{thm1} provides us with  
more precise information. It defines three $H$-equivariant 
maps
($\nu = 2 (p-1) \rho_H  - \l$) : 
\begin{align*}
 f_1 : \nabla(\l) \otimes \nabla^H(\nu)  &
\rightarrow {\rm End}_F(X_H, \{(X_i)_H\}),\\
 f_2 : \nabla(\l) \otimes \nabla^H(\nu) &
\rightarrow {\rm H}^0\big( X_H, \mathcal M_H(\nu) \big), \\
  f_3 :\nabla(\l) \otimes \nabla^H(\nu) & \rightarrow 
  k.
\end{align*}
related in the following way : let 
$s=f_2(x)$, for some $x \in   \nabla(\l) \otimes 
\nabla^H(\nu)$. Then $g \circ F_* s = f_1(x) $
defines a Frobenius splitting of $X_H$
up to a nonzero constant if and only
if $f_3(x)$ is nonzero. Moreover, the 
map $f_3$ may be explicitly described as the
composed morphism
$$  \nabla(\l) \otimes \nabla^H(\nu) \rightarrow 
 \nabla^H(\l_{|T_H}) \otimes \nabla^H(\nu) \xrightarrow{m}
 \nabla(2\rho_H(p-1)) \xrightarrow{{\rm ev}_{\nicefrac{H}{B_H}}} \kk,$$
induced by the restriction map  $\nabla(\l) \rightarrow 
 \nabla^H(\l_{|T_H})$, while $f_2$ is the tensorproduct 
of  (\ref{ind-theta0}) and (\ref{nu}).

It follows that 
we may choose $s$ to be a product of the  form $s_1 s_2$, 
where $s_1$ and $s_2$ denote global sections of the line bundles $\mathcal M_H$ and 
$\mathcal O\big( 2 (p-1) \rho_H  - \lambda_{|T_H} \big)$
respectively. Moreover,  $s_1$ can be chosen inside the image
of (\ref{ind-theta0}) or, even more specific, if $\nabla(\l)$
is generated by an element $v$ as an $H$-module then $s_1$ 
 can be chosen to be the image of $v$ under (\ref{ind-theta0}).
Part of the outcome of this is  that (\ref{Lsplit1})
factors through the morphism 
\begin{equation}
\label{Lsplit2}
F_* \mathcal O_{X_H}
\xrightarrow{F_* s_1} F_*\mathcal M_H
\end{equation}
and consequently $X_H$ also admits a Frobenius 
$\mathcal M_H$-splitting compatible with all 
the subvarieties $(X_i)_H$. Even though this 
conclusion seems weaker than Theorem \ref{thm1}
it will be useful for us later.
\end{rmk*}

\begin{cor}
Assume that $H$ is semisimple and simply connected and 
that  $(p-1) \rho_H - \l_{|T_H}$ is dominant as
a $T_H$-weight. Then there exists a $B_H$-equivariant map
$$ {\rm St}_H \otimes (p-1)\rho_H \rightarrow 
{\rm End}_F\big( X_H, \{(X_i)_H\} \big),$$
containing a Frobenius splitting in its image; i.e. 
 $X_H$ admits a $B_H$-canonical Frobenius 
splitting compatible splitting  each $(X_i)_H$
(see \cite{M}).
\end{cor}
\begin{proof}
With notation as in the remark above we have,
by assumption, a surjective multiplication 
map 
$$\nabla^H((p-1)\rho_H - \lambda_{|T_H}) 
\otimes \nabla^H((p-1)\rho_H) 
\rightarrow \nabla^H(\nu).$$
Composing this map with $f_1$ defines a
morphism 
$$  \nabla(\l) \otimes \nabla^H((p-1)\rho_H - \l_{|T_H} )
\otimes  {\rm St}_H 
\rightarrow {\rm End}_F(X_H, \{(X_i)_H\}),$$
and by the description of $f_3$ it suffices to find 
a $B_H$ semi-invariant element  $v$ in 
$$   \nabla(\l) \otimes \nabla^H((p-1)\rho_H - \l_{|T_H} ),$$
mapping to  a (nonzero) highest weight vector 
under the natural map
$$\nabla(\l) \otimes \nabla^H((p-1)\rho_H - \l_{|T_H} )
\rightarrow 
 \nabla^H((p-1)\rho_H ).$$
Let $v_-$ denote a lowest weight vector in 
$\nabla(\l)$ and $w_0^H$ denote the longest 
element in the Weyl group of $H$. Then $ w_0^H
v_-$ will be $B_H$ semi-invariant. Let now  $v^+$ denote 
a highest weight vector in $\nabla^H((p-1)\rho_H - \l_{|T_H} )$.
Then the product  $v = v^+ (w_0^H v_-)$ has the desired property.
\end{proof}

\begin{lem}
\label{lem1}
Assume that $X$ is contained in a $G$-variety $Z$ and 
that the line bundle $\mathcal M$ is the restriction of 
an ample $G$-linearized line bundle $\mathcal M_Z$ on
$Z$. Then $\mathcal M_H(\nu)$ is an ample line 
bundle for every regular dominant weight $\nu$ of 
$H$.
\end{lem}
\begin{proof}
Consider the $H$-equivariant morphism
$$ \psi : X_H =H \times_{B_H} X \rightarrow Z,$$
defined by 
$$ \psi(h,x) = h \cdot x.$$
The pull-back $\psi^*(\mathcal M_Z)$ is then 
an $H$-linearized line bundle on $X_H$. 
Consider $X$ as a $B_H$-stable subvariety of
$X_H$ in the natural way. Then, by assumption,
the restriction of $\psi^*(\mathcal M_Z)$ to $X$
coincides with the $B_H$-linearized line bundle
$\mathcal M$ on $X$. In particular, $\mathcal M_H$
must coincide with  $\psi^*(\mathcal M_Z)$ as 
$H$-linearized line bundles. 

As $Z$ is an $H$-variety we have an identification 
\begin{equation} 
\label{id2}
H \times_{B_H} Z \simeq \nicefrac{H}{B_H} \times Z,
\end{equation}
$$(h,z) \mapsto (h B_H,  h \cdot z).$$
Moreover, by the assumptions on $\nu$ and $\mathcal M_Z$, 
the external tensor product 
$\mathcal L_H(\nu) \boxtimes \mathcal M_Z$
is an ample line bundle on $ \nicefrac{H}{B_H} \times Z$.
The conclusion now follows as $\mathcal M_H(\nu)$ 
is the pull-back of $\mathcal L_H(\nu) \boxtimes \mathcal M_Z$
under the closed inclusion
$$  H \times_{B_H} X \subseteq H \times_{B_H} 
Z,$$
composed with the identification (\ref{id2}) above
\end{proof}

\begin{cor}
\label{cor-1}
Consider a setup as in Theorem \ref{thm1}.
Assume, moreover, that $2 (p-1) \rho_H  - \lambda_{|T_H}$
is a regular weight of $H$ and that $\mathcal M$ is the restriction 
of a $G$-linearized ample line bundle on a $G$-variety 
which contains  $X$ as a closed $B$-stable subvariety. If 
$X$ is a strongly $F$-regular projective variety then $H \times_{B_H} X$ 
is globally $F$-regular. 
\end{cor}
\begin{proof}
By Lemma \ref{lem1} and Theorem \ref{thm1} we know that 
$X_H$ admits a Frobenius splitting along an ample divisor. 
Moreover, as $X$ is strongle $F$-regular the same is the case
for $X_H$. Now apply Lemma \ref{global-lem}.
 \end{proof}

\section{Surjectivity condition}

In this section, we discuss one of the conditions in Thereom
\ref{thm1} about the surjectivity of the restriction map 
\begin{equation}
\label{surj}
\nabla(\l) \to \nabla^H(\l_{|T_H} ).
\end{equation}
The first observation is the following 

\begin{lem}
Assume that the restriction map (\ref{surj}) is surjective 
for all the fundamental weights $\omega_i$. Then the restriction 
map (\ref{surj}) is surjective for any dominant weight $\lambda$.
\end{lem}
\begin{proof}
We may assume that $\l=\sum_{i \in I} m_i \o_i$, where $m_i \ge 0$. Consider the following commutative diagram 
\[\xymatrix{
\nabla(\o_1)^{\otimes m_1} \otimes \cdots \otimes \nabla(\o_n)^{\otimes m_n} \ar[d]_-{f_1^{\otimes m_1} \otimes \cdots \otimes f_n^{\otimes m_n}} \ar[r]^-m & \nabla(\l) \ar[d]^-f  \\ \nabla^H(\o_1 \mid_{T_H} )^{\otimes m_1} \otimes \cdots \otimes \nabla^H(\o_n \mid_{T_H})^{\otimes m_n} \ar[r]^-{m'} & \nabla^H(\l \mid_{T_H}),}
\] where $f, f_1, \cdots, f_n$ are restriction maps and $m, m'$ are
multiplication maps. By \cite[Theorem 3]{RR}, $m$ and $m'$ are
surjective and by assumption the left vertical map is also surjective.
This proves the claim.
\end{proof}

The above results applies in case $\nabla^H({\o_1} \mid_{T_H}),
\cdots, \nabla^H({\o_n} \mid_{T_H} )$ are  irreducible $H$-modules. 
In  particular, it applies in case the characteristic $p$
of $\kk$ is sufficiently large; e.g. 

\begin{lem}
Assume that $\langle \rho_H+\o_i \mid_{T_H}, \b^\vee \rangle \le p$ 
for all fundamental weights $\o_i$ of $G$ and any positive root 
$\b$ of $H$. Then the restriction map (\ref{surj}) 
is surjective for all dominant weights $\l$ of $T$. 
\end{lem}

\subsection{}\label{Donkin} Below we give a criterion on the
surjectivity of (\ref{surj})  valid for all characteristics. 
We first recall the definition and some known results on Donkin pairs. 

An ascending chain $$0=V_0 \subset V_1 \subset V_2 \cdots \subset V_n=V$$ of submodules of a $G$-module $V$ is called a {\it good filtration} if for any $i$, $V_i/V_{i-1} \cong \oplus_\l \nabla(\l) \otimes_\kk A(\l, j)$ for some trivial $G$-modules $A(\l, j)$, where $\l$ runs over the set of dominant integral weights of $T$. 

We say that $(G, H)$ is a {\it Donkin pair} if for any $G$-module $M$ with a good filtration, the $H$-module $\res^G_H(M)$ also has a good filtration. The following are some examples of Donkin pairs that will be used in this paper :

(1) If $H$ is a Levi subgroup of $G$, then $(G, H)$ is a Donkin pair. This is proved by Donkin in \cite{D} for almost all cases and later by Mathieu in \cite{M} in full generality.

(2) If $H$ is the centralizer of a graph automorphism of $G$ or the centralizer of an involution of $G$ and the characteristic of $\kk$ is at least $3$, then $(G, H)$ is a Donkin pair. This is conjectured by Brundan in \cite{B} and proved by Van der Kallen in \cite{V}. 

\begin{lem}\label{surj-Don}
Let $(G, H)$ be a Donkin pair. Let $\l \in X^*(T)$ be a dominant weight. Then the restriction map $\nabla(\l) \to \nabla^H(\l_{|T_H} )$ is surjective for all dominant weights $\l$ of $T$. 
\end{lem}

\begin{proof} By definition, $\nabla(\l)$ has a good filtration. Hence $\res^G_H \nabla(\l)$ also has a good filtration
$$  0=M_0 \subset \cdots \subset M_{n-1} \subset M_n=\res^G_H \nabla(\l).$$
As $\res^G_H \nabla(\l)$ is finite dimensional we may furthermore assume that the 
quotients $M_i/M_{i-1}$ are isomorphic to $\nabla^H(\nu_i)$ for certain dominant 
$T_{H}$-weights $\nu_i$. Moreover, as $-\lambda$ is the (unique) lowest weight  
vector of  $\nabla(\l)$ we must have $\nu_i \leq \lambda_{|T_H}$. 

Now use that $\nabla(\l) \to \nabla^H(\l_{|T_H} )$ is nonzero to find a minimal $j$  such
that the induced map $M_j \rightarrow  \nabla^H(\l_{|T_H} )$ is  nonzero. In particular,
we obtain a nonzero map 
\begin{equation}
\label{good}
\nabla^H(\nu_j) \rightarrow \nabla^H(\l_{|T_H}).
\end{equation}
By Frobenius reciprocity this implies that $\nu_j \geq \lambda_{|T_H}$ and thus 
that $\nu_j = \lambda_{|T_H}$. Another use of Frobenius 
reciprocity  now implies that (\ref{good}) is the identity map which suffices
to end the proof.
 \end{proof}

\section{Frobenius splitting of $\nicefrac {P_J}{B}$}
\label{FsplitPB}

We eventually want to apply the result in Section \ref{FsplitX} to the
case where $X$ is $B$-variety of the form $\nicefrac{P_J} {B}$. In particular, we
need a good description of the Frobenius splitting properties of $\nicefrac {P_J}{B}$.
As a variety $\nicefrac{P_J}{B}$ is just the flag variety associated to the Levi subgroup 
$L_J$ and as such we already have a detailed knowledge about its Frobenius 
splitting properties. The aim of this section is to formulate this knowledge
in a $B$-equivariant way. 

\begin{lem}
\label{st-decomp}
$$\sum_{\alpha \in R_J^+} \alpha = \rho_J - w_0^J \rho_J.$$
\end{lem}
\begin{proof}
Define
$$\lambda^J = \sum_{\beta \in R^+ \setminus R_J^+} \beta, ~ ~ \  \ ~ \lambda_J = \sum_{\alpha \in R_J^+} \alpha .$$
Then $2 \rho = \lambda^J + \lambda_J$. Now recall the following 
identities 
$$ w_0^J(R_J^+) = -R_J^+,$$
$$ w_0^J(R^+ \setminus R_J^+) = R^+ \setminus R_J^+.$$
It follows that 
$$ 2w_0^J \rho = w_0^J (\lambda^J + \lambda_J) = \lambda^J - \lambda_J.$$
and thus 
$$ \rho - w_0^J \rho=  \lambda_J.$$
On the other hand 
$$ w_0^J \rho = w_0^J (\rho_J + \rho_{I \setminus J}) = w_0^J \rho_J 
+ \rho_{I \setminus J},$$
and thus 
$$ \lambda_J= \rho - w_0^J \rho = \rho_J - w_0^J \rho_J.$$
\end{proof}

The trivial $P_J$-linearization on $\mathcal O_{\nicefrac{P_J}{B}}$ 
induces a $P_J$-linearization of the line bundle $\omega_{\nicefrac{P_J}{B}}$
which is associated to $B$-character 
$$ \sum_{\alpha \in R_J^+} \alpha.$$
In particular, we have a $P_J$-equivariant identification
\begin{equation}
\label{eq1}
\nabla^J\big( (p-1) \sum_{\alpha \in R_J^+} \alpha \big) =
{\rm H}^0 \big( \nicefrac{P_J}{B}, \omega_{\nicefrac{P_J}{B}}^{1-p} \big) 
 \simeq {\rm End}_F(\nicefrac{P_J}{B}).
\end{equation}
By Lemma \ref{st-decomp} this leads to the following 
central morphism
\begin{prop}
There exists a surjective $P_J$-equivariant morphism 
\begin{equation}
\label{st}
\nabla^J((p-1) \rho_J) \otimes \nabla^J( (1-p)w_0^J \rho_J) 
\rightarrow   {\rm End}_F(\nicefrac{P_J}{B}).
\end{equation}
Composing (\ref{st}) with the evaluation map defines 
a $P_J$-equivariant map
$$ \nabla^J((p-1) \rho_J) \otimes \nabla^J((1-p) w_0^J \rho_J) 
\rightarrow     k,$$
which defines an $P_J$-equivariant isomorphism 
\begin{equation}
\label{eq2}
  \nabla^J((p-1) \rho_J) \rightarrow 
 \nabla^J((1-p) w_0^J \rho_J) ^*,
\end{equation}
between irreducible $P_J$-representations.
\end{prop}
\begin{proof}
The map (\ref{st}) is just the surjective multiplication map 
$$\nabla^J((p-1) \rho_J) \otimes \nabla^J((1-p) w_0^J \rho_J) 
\rightarrow \nabla^J((p-1) (\rho_J-w_0^J \rho_J)), $$
composed with the identification (\ref{eq1}) using
Lemma \ref{st-decomp}.

For the second part, it suffices to prove the irreducibility
claim and that  (\ref{eq2}) is non-zero.  That (\ref{eq2})
is non-zero follows as (\ref{st}) is surjective and as $\nicefrac{P_J}{B}$  
admits a Frobenius splitting. Consider next the simply connected 
commutator group $G_J = (L_J, L_J)$.
As $G_J$-modules both tensor-factors
on the left hand side of (\ref{st}) are equal 
to the associated Steinberg module ${\rm St}_J$
of $G_J$. In particular, they are irreducible as 
$G_J$-modules and thus also as $P_J$-modules.
\end{proof}

Let $v_J^+$ denote a highest weight vector of 
the $P_J$-module $\nabla^J((1-p) w_0^J \rho_J)$.
The weight  of $v_J^+$ is then $(p-1) \rho_J$. The
element $v^+_J$
vanishes with multiplicity $(p-1)$ along the union of 
the codimension 1 Schubert varieties in $\nicefrac{P_J}{B}$. 
In particular, we obtain (cf. \cite[Prop.1.3.11]{BK})

\begin{cor}
\label{compatible}
The restriction of (\ref{st}) to the highest weight space 
defines a $B$-equivariant morphism 
$$\nabla^J((p-1) \rho_J) \otimes  \kk_{(p-1)\rho_J}
\rightarrow   {\rm End}_F(\nicefrac{P_J}{B}, \{ X(w) \}_{w \in W_J}),$$
where the image is compatible with all Schubert varieties contained 
in $\nicefrac{P_J}{B}$ and contains a Frobenius splitting of $\nicefrac{P_J}{B}$.
\end{cor}

\subsection{Frobenius $\mathcal M$-splitting}

We now want to formulate a slightly more precise statement 
based on the observations above. Define $\mathcal M = 
\mathcal L_J((p-1)\rho_J)$ and start by observing 
the $P_J$-equivariant identification
\begin{equation}
\label{Msplit}
{\rm End}^{ \mathcal M}_F\big(
\nicefrac{P_j}{B} \big) \simeq 
{\rm H}^0 \big( \nicefrac{P_J}{B}, \omega_{\nicefrac{P_J}{B}}^{(1-p)} 
\otimes \mathcal M^{-1} \big) \simeq
\nabla^J\big (    (1-p)w_0^J \rho_j \big)   .
\end{equation}
As
\begin{equation}
\label{iden}
{\rm H}^0\big( 
\nicefrac{P_J}{B}, \mathcal M\big) \simeq
\nabla^J \big( (p-1)  \rho_J \big),
\end{equation}
we may relate  (\ref{st}) with (\ref{Msplit}) by the natural map
$$ {\rm H}^0\big( 
\nicefrac{P_J}{B}, \mathcal M\big) \otimes 
{\rm End}^{\mathcal M}_F\big(\nicefrac{P_J}{B}\big)
\rightarrow {\rm End}^{ }_F\big(\nicefrac{P_J}{B}\big).$$
The statement in Corollary \ref{compatible} then means that
the highest weight line in ${\rm End}^{\mathcal M}_F(\nicefrac{P_J}{B})$
is compatible with any Schubert variety $X(w)$, $w \in W_J$.
More precisely we find the following result related to 
the setup in Section \ref{FsplitX} :
\begin{prop}
\label{prop1}
There exists a $B$-equivariant map
\begin{equation}
\label{eq3}
\theta_J :  \kk_{(p-1) \rho_J} \rightarrow {\rm End}^{\mathcal M}_F\big(\nicefrac{P_J}{B}
, \{ X(w) \}_{w \in W_J}\big),
\end{equation}
such that when
$$\phi_J : \nabla\big((p-1) \rho_J \big) \rightarrow \nabla^J \big( (p-1)  \rho_J \big)  \simeq 
{\rm H}^0\big( \nicefrac{P_J}{B}, \mathcal M\big),
$$
denotes the restricition map, then the induced $B$-equivariant map
$$ (\theta_J, \phi_J) :  \nabla\big((p-1) \rho_J \big)  
\otimes  \kk_{(p-1) \rho_J} \rightarrow  {\rm End}^{ 
}_F\big(\nicefrac{P_J}{B}
, \{ X(w) \}_{w \in W_J} \big),$$
contains a Frobenius splitting of $\nicefrac{P_J}{B}$ in its image. 
\end{prop}

\section{Frobenius splitting $H \times_{B_H} X(w)$.}

We are now ready to apply the results in Section \ref{FsplitX} and 
Section \ref{FsplitPB}. 

\begin{thm}
\label{thm2}
Let $\mathcal M$ denote the line bundle $\mathcal L_J((p-1)\rho_J)$
on $\nicefrac{P_J}{B}$. If the following conditions are satisfied 
\begin{enumerate}
\item The $T_H$-weight $2 \rho_H  - \rho_J \mid_{T_H} $ is dominant,
\item The $H$-equivariant restriction morphism 
$$  \nabla \big( (p-1) \rho_J \big)  \rightarrow \nabla^H \big( (p-1) \rho_J \mid_{T_H} \big),$$
is surjective,
\end{enumerate}
then the variety $H \times_{B_H} {\nicefrac{P_J}{B}}$ admits a 
Frobenius $\mathcal M_H((p-1) (2 \rho_H - \rho_J \mid_{T_H}))$-splitting 
which is compatible with the subvarieties 
$H \times_{B_H} X(w)$, $w \in W_J$.  
If, moreover,  $2 \rho_H  - \rho_J $ is regular then the line bundle
$\mathcal M_H((p-1) (2 \rho_H - \rho_J \mid_{T_H}))$ is ample and as a 
consequence the varieties $H \times_{B_H} X(w)$, $w \in W_J$,
are globally $F$-regular. 
\end{thm}
\begin{proof}
The first part of the statement follows by an application of 
Theorem \ref{thm1} and Proposition \ref{prop1}. The 
second part  follows from Corollary \ref{cor-1} by using 
that Schubert varieties are strongly $F$-regular.
\end{proof}

By applying the natural morphism
\begin{equation}
\label{nat}
\pi_J : H \times_{B_H} \nicefrac{P_J}{B} \rightarrow \nicefrac{HP_J}{B}\subseteq
\nicefrac{G}{B},
\end{equation}
we may sometimes transfer the statements in Theorem \ref{thm2}
into statements about $H$-orbit closures in $\nicefrac{G}{B}$. For 
this to work we however need (\ref{nat}) to be separable, which 
is easily seen to be equivalent to the following condition on the 
level of the Lie algebras
\begin{equation}
\label{lie}
{\rm Lie}(H) \cap {\rm Lie}(P_J) =  {\rm Lie}(H \cap P_J).
\end{equation}

\begin{cor}
\label{cor1}
Assume that the relation (\ref{lie}) as well as condition (1) and (2) of 
Theorem \ref{thm2} are satisfied. Then $\nicefrac{HP_J}{B}$
admits a Frobenius $\mathcal L((p-1)\rho_J)_{|\nicefrac{HP_J}{B}}$-splitting which 
is compatible 
with all subvarieties of the form ${H \cdot  X(w)} $, $w \in W_J$. 
If, moreover,  $2 \rho_H  - \rho_J \mid_{T_H}$ is regular then each 
$H \cdot X(w)$, $w \in W_J$, is globally $F$-regular. 
\end{cor}
\begin{proof}
Use the notation $\mathcal M$ to denote the line bundle
 $\mathcal L_J((p-1)\rho_J)$.
Applying the remark in Section \ref{FsplitX} we find, as in 
Theorem \ref{thm2}, that the variety $H \times_{B_H} 
{\nicefrac{P_J}{B}}$ admits a Frobenius $\mathcal M_H$-splitting 
which is compatible with the subvarieties $H \times_{B_H} X(w)$, 
$w \in W_J$; i.e. there exists a map
\begin{equation}
\label{comp}
 F_* \mathcal M_H \rightarrow  \mathcal O_{H \times_{B_H} 
{\nicefrac{P_J}{B}}},
\end{equation}
compatible with all subvarieties $H \times_{B_H} X(w)$, 
$w \in W_J$, and a global section $s$ of $\mathcal M_H$  
such the composition of (\ref{comp}) with 
\begin{equation}
\label{comp1}
 F_* \mathcal O_{ H \times_{B_H} 
{\nicefrac{P_J}{B}}} \xrightarrow{F_* s} 
 F_* \mathcal M_H,
\end{equation}
defines a Frobenius splitting of 
$H \times_{B_H} 
{\nicefrac{P_J}{B}}$. As observed in the remark in 
Section \ref{FsplitX} we may even assume that $s$ 
is contained  in the image of the morphism 
\begin{equation}
\label{s} 
\nabla((p-1) \rho_J) \rightarrow {\rm Ind}_{B_H}^H 
\big({\rm H}^0 ( \nicefrac{P_J}{ B}, \mathcal M ) \big)
=  {\rm Ind}_{B_H}^H \big( \nabla^J((p-1) \rho_J) \big),
\end{equation}
i.e. we may assume that $s$ is the pull-back $\pi_J^*(s')$
of a global section $s'$ of the line bundle $\mathcal L\big((p-1) \rho_J
\big)_{|\nicefrac{HP_J}{B}}$. In this connection we notice 
the identity $\mathcal M_H \simeq \pi_J^*\big (\mathcal L\big( (p-1) \rho_J 
\big) \big)$ which follows by arguing as in the proof of Lemma \ref{lem1}.

We now claim that the morphism 
\begin{equation}
\label{claim}
 \mathcal O_{\nicefrac{HP_J}{B}} \rightarrow 
 (\pi_J)_* \mathcal O_{H \times_{B_H} \nicefrac{P_J} {B}},
\end{equation}
induced by $\pi_J$, is an isomorphism. If so, we may apply 
$(\pi_J)_*$ to the composition of (\ref{comp}) and (\ref{comp1}) 
to obtain a map
$$ F_*  \mathcal O_{\nicefrac{HP_J}{B}}  \xrightarrow{F_* s'}
 F_* \mathcal L\big((p-1) \rho_J \big)_{|\nicefrac{HP_J}{B}} 
\rightarrow  \mathcal O_{\nicefrac{HP_J}{B}},$$
defining a Frobenius $\mathcal L\big((p-1) \rho_J \big)_{|\nicefrac{HP_J}{B}}$-splitting
of $\nicefrac{HP_J}{B}$ compatible with all the subvarieties (cf. \cite[Lemma 1.1.8] {BK})
$$\nicefrac{HBwB}{B} = \pi_J\big(H \times_{B_H} X(w) \big), w \in W_J.$$
This is exactly the first part of the statement of this corollary. 

To prove the claim consider the natural morphism
\begin{equation} 
\label{closed}
\nicefrac{H}{P_J \cap H} \rightarrow \nicefrac{G}{P_J}.
\end{equation}
As relation (\ref{lie}) is assumed to be satisfied the morphism 
(\ref{closed}) is a closed embedding. In particular, the pull 
back $ \nicefrac{HP_J}{B}$ of the closed subvariety 
$\nicefrac{H}{P_J \cap H}$ by the morphism 
$$ \nicefrac{ G}{B} = G \times_{P_J} \nicefrac{P_J}{B}
\rightarrow \nicefrac{G}{P_J},$$
is isomorphic to $H \times_{P_J \cap H} \nicefrac{P_J}{B} $.
The claim now follows as the natural morphism 
$$ H \times_{B_H} \nicefrac{P_J}{B} \rightarrow
H \times_{P_J \cap H} \nicefrac{P_J}{B},$$
is a locally trivial $\nicefrac{P_J \cap H}{B_H}$-bundle. 
This ends the proof of the first part of the statement.

As to the second statement the global $F$-regularity of 
$\nicefrac{HP_J}{B}$ follows from Theorem \ref{thm2}  
using the claim above 
and the fact that global $F$-regularity is preserved by 
push forward (Lemma \ref{glob-push}). The global $F$-regularity of 
$H \cdot X(w)$, $w \in W_J$, now follows 
in the same way by applying Lemma \ref{pushforward}.
\end{proof}

\begin{rmk*}
In case $\nicefrac{G}{B}$ contains a dense $H$-orbit 
$H gB$ and $\nabla\big({(p-1)\rho_J}\big)$ is irreducible 
as a $G$-module we may, in the proof of the above 
corollary, choose the section $s$ to be the image under 
(\ref{s}) of the element $g v_+$. Here $v_+$ denotes a 
the highest weight vector in $\nabla\big({(p-1)\rho_J}\big)$.
This follows from the remark in Section \ref{FsplitX}
as $g v_+$ generates $\nabla \big({(p-1)\rho_J}\big)$
as an $H$-module. In this case the zero divisor associated 
to $s'$ is the sum 
$$ D= (p-1) \sum_{i \in J} \big(g \cdot X(w_0 s_i) \cap \nicefrac{H P_J}{B} \big). $$
The Frobenius splitting in Corollary \ref{cor1} may therefore
also be considered as a Frobenius $D$-splitting
(cf. \cite[$\S$ 1.4]{BK}).

\end{rmk*}

\begin{rmk*}
The assumption in Corollary \ref{cor1} that the 
relation (\ref{lie}) is satisfied is necessary to 
make the proof work. This assumption does not follow from 
the rest of the assumptions as can be seen by the 
following example : Consider 
 the case $G=SL_2 \times SL_2$
and $H=\{(g, F(g)); g \in SL_2\} \subset G$.
If $P_J$ is chosen to be the set of pair $(g_1, g_2) \in G $ 
with the condition that $g_2$ upper triangular, then the 
natural morphism (\ref{closed}) is not a closed embedding.
However, the rest of the assumption in Corollary \ref{cor1}
are satisfied. 
\end{rmk*}

As relation (\ref{lie}) is always satisfied in case $J=I$ we find 

\begin{cor}\label{HX3}
If $2 \rho_H-\rho_{|T_H}$ is a dominant $T_H$-weight and if  the restriction map
$$  \nabla \big( (p-1) \rho \big)  \rightarrow \nabla^H \big( (p-1) \rho_{|T_H}\big),$$
is surjective, then there exists a Frobenius $\mathcal L\big ( (p-1) \rho)$-splitting of
$\nicefrac{G}{B}$ that is compatibly with all subvarieties over the form  $H \cdot X(w)$,  $w \in W$. 
In particular,  for any dominant weight $\l$ of $T$ and any $w \in W$ we
have
\begin{enumerate}
\item ${\rm H}^i(H \cdot X(w), \cl(\l))=0$ for $i \ge 1$.
\item The restriction map
$${\rm H}^0(\nicefrac{G}{B}, \cl(\l)) \to {\rm H}^0(H \cdot X(w), \cl(\l)),$$
 is surjective. 
\end{enumerate}

\end{cor}

%\subsection{}  We call a triple $(G, H, B)$ {\it admissible} if $\nicefrac {G}{H}$ is a spherical variety, $(G, H)$ is a Donkin pair and $2 \rho_H-\rho$ is a dominant $T_H$-weight. 

\begin{cor}
We keep the assumption in Corollary \ref{HX3}. Assume furthermore that $H$ is a spherical subgroup of $G$. Let $w \in W$ and $g \in G$ such that $H g B/B$ is open dense in ${H \cdot X(w)}$. Let $\l$ be a dominant weight of $T$ and $V(\l)$ be the Weyl module $\nabla (\l)^*$. Let $v_\l$ be a nonzero vector of weight $\l$ in $V(\l)$. Then ${\rm H}^0\big({H\cdot  X(w)}, \cl(\l)\big)^*$ is isomorphic to the $H$-submodule of $V(\l)$ generated by $g v_\l$. 
\end{cor}

\begin{proof}
We follow the idea of the proof of  \cite[Cor.3.3.11]{BK}.
Let $M$ be the $H$-module generated by $g v_\l$. By Corollary \ref{HX3}, the restriction map 
$$\g : \nabla \big(\l \big)\to {\rm H}^0\big({H \cdot X(w)}, \cl(\l)\big),$$
is surjective. As $H g B/B$ is dense in $H \cdot X(w)$ we have
\begin{equation}
\label{kernel}
\begin{split}
\ker(\g) & =\{f \in \nabla (\l) : f_{\mid HgB/B} =0\} \\
&=\{f \in V(\l)^*:  ((g^{-1} h) f)(eB) =0, \text{ for all $h \in H$}\}.
\end{split}
\end{equation}
The central point is now that the $B$-equivariant map
$ \nabla \big(\l \big)\to \kk_{-\l},$
coincides with the map $f \mapsto f(eB)$. In particular,
$f$ is zero at $eB$ if and only if $v_\l(f) = 0$. Thus,
by (\ref{kernel}), $f$ is contained in $\ker(\g)$ if and 
only if $ ((hg)v_\lambda)(f)$ is zero for all $h \in H$. 
The kernel of $\g$ is therefore $(V(\l)/M)^*$. This
ends the proof.
\end{proof}

\section{Examples}

\subsection{}\label{Ex1} In this subsection, we discuss some cases where there is a splitting on $\nicefrac{G}{B} $ that is compatible with all the $H$-orbit closures. Let $(G, H)$ be one of the following: $(H, H)$, $(H \times H, H_{\rm diag})$, $(A_{2n+1}, C_n)$, $(D_n, B_{n-1})$, $(E_6, F_4)$, $(B_3, G_2)$. Let $B$ be a Borel subgroup of $G$ such that $B_H=H \cap B$ is a Borel subgroup of $H$. It is known that all such $B$ are conjugated by $H$ (see \cite[Prop 2.2]{Re}) and that all the $H$-orbit closure in $\nicefrac{G}{B}$ are of the form $\nicefrac{ \overline{H B w B}}{B}$ for some $w \in W$. By $\S$\ref{Donkin}(2), $(G, H)$ is a Donkin pair. It is also easy to check that $2 \rho_H-\rho \mid_{T_H}$ is dominant for $B_H$. Hence $(G, H, B)$ is admissible. By Corollary \ref{HX3}, there exists a  Frobenius $\mathcal L \big( (p-1) \rho \big)$-splitting  
of $\nicefrac{G}{B}$ that is compatible with all $H$-orbit closures. 

It is proved by N. Ressayre in \cite[Theorem A]{Re} that if $G$ is a complex reductive group and $H$ a closed quasi-simple subgroup of $G$, then $(G, H)$ is of minimal rank if and only if $(G, H)$ is (up to isomorphism) one of the pairs above. 

However, in positive characteristic, the pair $(G, H)$ is also of minimal rank, where $G=SL_2 \times SL_2$ and $H=\{(g, F(g)); g \in SL_2\} \subset G$. The closed $H$-orbit in $G/B \cong \mathbb{P}^1 \times \mathbb{P}^1$ is defined by the equation $X^p W-Y^p V=0$, where $X,Y$ are the coordinates of the first $\mathbb{P}^1$ and $V, W$ are the coordinates of the second $\mathbb{P}^1$. There is no Frobenius splitting on $\nicefrac{G}{B}$ that compatibly splits the closed $H$-orbit. 

\subsection{}\label{Ex2} Let $G=Sp_4$ and $B$ a Borel subgroup of $G$. Let $\a$ be the short simple root and $\b$ be the long simple root. We denote by $s_1$ and $s_2$ the simple reflections corresponding to $\a$ and $\b$ respectively. Let $H$ be the standard Levi subgroup corresponding to $\a$. The $H$-orbit closures on the flag variety of $G$ are described in the the following graph

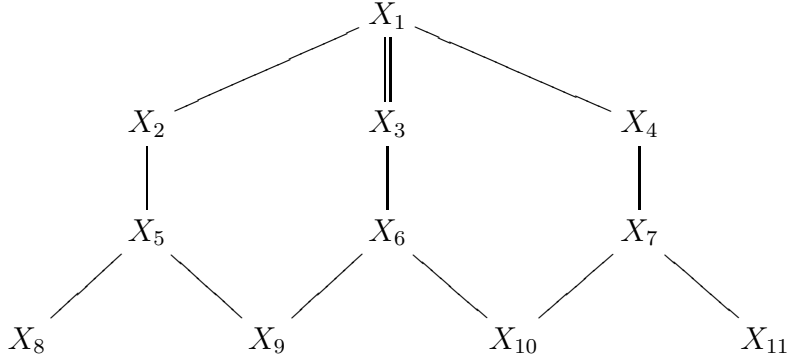
\[\xymatrix{& & & X_1 \ar@{-}[lld] \ar@{=}[d] \ar@{-}[rrd] & & & \\ & X_2 \ar@{-}[d] & & X_3 \ar@{-}[d] & & X_4 \ar@{-}[d] \\ & X_5 \ar@{-}[ld] \ar@{-}[rd] & & X_6 \ar@{-}[ld] \ar@{-}[rd] & & X_7 \ar@{-}[ld] \ar@{-}[rd] \\ X_8 & & X_9 & & X_{10} & & X_{11}}\]

Here $X_8$ is the $H$-orbit of $B_1=B$, $X_9$ is the $H$-orbit of $B_2=s_2 B s_2$, $X_{10}$ is the $H$-orbit of $B_3=s_2 s_1 B s_1 s_2$ and $X_{11}$ is the $H$-orbit of $B_4=s_2 s_1 s_2 B s_2 s_1 s_2$. It is easy to see that $2 \rho_H-\rho \mid_{T_i \cap H}$ is dominant for $B_1 \cap H$ and $B_4 \cap H$ but is not dominant for $B_2 \cap H$ and $B_3 \cap H$. By Corollary \ref{cor1}, 

(1) There exists a  Frobenius $\mathcal L\big((p-1)\rho \big)$-splitting of the flag variety $\nicefrac{G}{B}$
 that  compatibly splits $X_2$, $X_5$ and $X_8$. We may even apply the second part of 
Corollary \ref{cor1} to obtain global $F$-regularity of $X_2$, $X_5$ and $X_8$.
In fact, $X_2$, $X_5$ and $X_8$ is just a subset of
the set of Schubert varieties so this is a well known result. 

(2) Similarly, there exists a Frobenius splitting of the flag variety $\nicefrac{G}{B_4}$ that 
compatible splits certain $H$-orbit closures $X_4'$, $X_7'$ and $X_{11}'$ (which are also
Schubert varieties in $\nicefrac{G}{B_4}$). By the natural 
identification of $\nicefrac{G}{B_4}$ with $\nicefrac{G}{B}$ this leads to a Frobenius 
 $\mathcal L\big((p-1)\rho \big)$-splitting of the flag variety $\nicefrac{G}{B}$
 that  compatibly splits $X_4$, $X_7$ and $X_{11}$. The 
varieties  $X_4$, $X_7$ and $X_{11}$ are not Schubert varieties in $\nicefrac{G}{B}$.

(3) Let $P_1$ and $P_2$ denote the minimal parabolic subgroups containing 
$B_2$. Fix notation such that $P_1$ corresponds to the short simple
root. Then, by Corollary \ref{cor1}, the variety $\nicefrac{H P_2}{ B_2}$ in 
$\nicefrac{G}{B_2}$  admits a Frobenius $\mathcal M_2$-splitting compatible 
with the orbit closure  $\nicefrac{H B_2}{ B_2}$. Here $\mathcal M_2$
is some ample line bundle on $\nicefrac{H P_2}{ B_2}$ which can be explicitly 
determined. In this case we cannot apply the strong part of Corollary \ref{cor1}.
Focusing on $P_1$ instead  we may conclude that $\nicefrac{H P_1}{ B_2}$ 
admits a Frobenius $\mathcal M_1$-splitting compatible with  $\nicefrac{H B_2}{ B_2}$.
Again $\mathcal M_1$ is some ample line bundle. In this case we may apply
the strong part of Corollary \ref{cor1} to obtain global $F$-regularity of 
both $\nicefrac{H P_1}{ B_2}$ and  $\nicefrac{H B_2}{ B_2}$. Transferring 
this information into $\nicefrac{G}{B}$ it means that $X_6$ as well as $X_5$
admits a Frobenius splitting, along an ample line bundle, which is compatible 
with $X_9$. Moreover, as $X_6$ corresponds to $\nicefrac{H P_1}{ B_2}$ 
we may conclude global $F$-regularity of  $X_6$. Notice that $X_6$ and 
 $X_9$ are not multiplicity-free in the sense of \cite{Br}. 

(4) Similar statement as for the pairs $(X_5,X_9)$ and $(X_6, X_9)$ are 
also satisfied for the pairs $(X_7,X_{10})$ and $(X_6, X_{10})$.

We do not know if $X_3$ admits a Frobenius splitting.

\subsection{}\label{Ex3} Let $(G, H)=(SL_n, SO_n)$ and $p \ge 3$. Then $(G, H)$ is a Donkin pair. Let $I=\{1, 2, \cdots, n-1\}$ be the set of simple roots of $G$ and $$J=\begin{cases} I-\{\frac{n}{2}\}, & \text{ if } 2 \ \mid n; \\ I-\{\frac{n-1}{2}\} \text { or } I-\{\frac{n+1}{2}\}, & \text{ if } 2 \nmid n. \end{cases}$$ Then $2 \rho_H-\rho_J \mid_{T_H}$ is dominant. By Corollary \ref{cor1}, $\nicefrac{H P_J}{B}$ admits a Frobenius splitting along an ample divisor which is compatible with all subvarieties 
$\nicefrac{\overline{H B w B}}{B}$ for $w \in W_J$. 

Notice that none of the codimension one $H$-orbit closures in $\nicefrac{G}{B}$ are multiplicity-free.

\subsection{} Let $(G, H)=(H \times H \times H, H_{\rm diag})$. Then $\rho \mid_{T_H}=3 \rho_H$ and $2 \rho_H-\rho \mid_{T_H}$ is not dominant. Moreover, $(G, H)$ satisfies the pairing criterion (see \cite[Example 8]{V}). It is easy to see that the subvarieties $\nicefrac{\overline{H B ( w_0^H, 1, 1) B}}{B}, \nicefrac{\overline{H B (1,  w_0^H, 1) B}}{B}, \nicefrac{\overline{H B (1, 1,  w_0^H) B}}{B}$ are the partial diagonals of $\nicefrac{G}{B}=\nicefrac{H}{B_H} \times \nicefrac{H}{B_H} \times \nicefrac{H}{B_H}$. By \cite[Exercise 3.5.3]{BK},  the flag variety of $G$ does not admit a Frobenius splitting which is compatible with all the partial diagonals. This proves the necessity  in Corollary \ref{HX3} of something like the condition 
that $2 \rho_H-\rho \mid_{T_H}$ is dominant.

\bibliographystyle{amsalpha}

\end{document}